\documentclass[twoside]{amsart}
\usepackage{amsfonts,amsthm}
\usepackage{enumerate}
\theoremstyle{plain}
\newtheorem*{thm A}{Theorem~A}
\newtheorem*{thm B}{Theorem~B}
\newtheorem*{thm C}{Theorem~C}
\newtheorem*{thm D}{Theorem~D}
\newtheorem*{thm E}{Theorem~E}
\newtheorem*{Main Theorem}{Main Theorem}
\newtheorem*{proof**}{Proof of Theorem~\ref{Theorem 2.7}}

\newtheorem*{thm 1}{Theorem~1}
\newtheorem*{thm 2}{Theorem~2}
\newtheorem*{pro A}{Proposition~A}
\newtheorem*{pro B}{Proposition~B}
\newtheorem*{lem A}{Lemma~A}
\newtheorem*{lem B}{Lemma~B}
\newtheorem*{lem C}{Lemma~C}
\newtheorem*{lem D}{Lemma~D}

\newtheorem*{proof*}{\it The proof of Corollary}
\newtheorem*{exmp 1}{Example~1}
\newtheorem*{exmp 2}{Example~2}

\newtheorem{theorem}{Theorem}[section]
\newtheorem{corollary}{Corollary}[section]
\newtheorem{lemma}{Lemma}[section]
\newtheorem{proposition}[theorem]{Proposition}

\theoremstyle{plain}
\newcommand{\be}{\begin{equation}}
\newcommand{\ee}{\end{equation}}
\newcommand{\bea}{\begin{eqnarray}}
\newcommand{\eea}{\end{eqnarray}}
\newcommand{\ba}{\begin{array}}

\newcommand{\ea}{\end{array}}

\newcommand{\bc}{\begin{center}}
\newcommand{\ec}{\end{center}}

\newcommand{\benu}{\begin{enumerate}}
\newcommand{\eenu}{\end{enumerate}}
\newcommand{\bpr}{\begin{proposition}}
\newcommand{\epr}{\end{proposition}}
\newcommand{\ble}{\begin{lemma}}
\newcommand{\ele}{\end{lemma}}
\newcommand{\bco}{\begin{corollary}}
\newcommand{\eco}{\end{corollary}}

\def \D{\mathfrak D}

\begin{document}

\title[Parallel Ricci tensor in generalized Tanaka-Webster connection]
{Generalized Tanaka-Webster Parallel Ricci tensor in complex two-plane Grassmannians}
\vspace{0.2in}
\author[Juan de Dios P\'{e}rez and Young Jin Suh]{Juan de Dios P\'{e}rez and Young Jin Suh}
\address{\newline
Juan de Dios P\'{e}rez
\newline Departamento de Geometria y Topologia,
\newline Universidad de Granada,
\newline 18071-Granada, Spain}
\email{jdperez@ugr.es}

\address{\newline
Young Jin Suh
\newline Kyungpook National University,
\newline Department of Mathematics,
\newline Taegu 702-701, Korea}
\email{yjsuh@knu.ac.kr}


\thanks{* First author is partially supported by MCT-FEDER Grant MTM2010-18099 and second author by Proj. No. NRF-2012-R1A2A2A-01043023 from National Research Foundation}

\begin{abstract}
We prove the non-existence of Hopf real hypersurfaces in complex two-plane Grassmannians
whose Ricci tensor is parallel with respect to the generalized Tanaka-Webster
connection.
\end{abstract}

\maketitle

\section{Introduction}
\setcounter{equation}{0}
\renewcommand{\theequation}{1.\arabic{equation}}
\vspace{0.13in}

The generalized Tanaka-Webster connection (from now on, g-Tanaka Webster connnection) for contact metric manifolds was introduced by Tanno (\cite{TAN}) as a generalization of the connection defined by Tanaka in \cite{TA} and, independently, by Webster in \cite{W}. The Tanaka-Webster connection is defined as a canonical affine connection on a non-degenerate, pseudo-Hermitian CR-manifold. A real hypersurface $M$ in a K\"{a}hler manifold has an (integrable) CR-structure associated with the almost contact structure $(\phi ,\xi ,\eta ,g)$ induced on $M$ by the K\"{a}hler structure, but, in general, this CR-structure is not guaranteed to be pseudo-Hermitian. Kon considered
a g-Tanaka-Webster connection for a real hypersurface of a K\"{a}hler manifold  (see \cite{Kon}) by

\begin{equation} \label{1.1}
\hat{\nabla}_X^{(k)}Y=\nabla_XY+g(\phi AX,Y)\xi -\eta(Y)\phi AX-k\eta(X)\phi Y
\end{equation}

\noindent for any $X,Y$ tangent to $M$, where $\nabla$ denotes the Levi-Civita connection on $M$, $A$ is the shape operator on $M$ and $k$ is a non-zero real number. In particular, if the real hypersurface satisfies $A\phi +\phi A=2k\phi$, then the g-Tanaka-Webster connection $\hat {\nabla}^{(k)}$ coincides with the Tanaka-Webster connection (see \cite{TAN} and \cite{W}).
\par
\vskip 6pt
Let us denote by $G_2(\mathbb{C}^{m+2})$ the set of all complex 2-dimensional linear subspaces in $\mathbb{C}^{m+2}$. It is known to be the unique compact irreducible Riemannian symmetric space equipped with both a K\"{a}hler structure $J$ and a quaternionic K\"{a}hler structure $\mathfrak{J}$ not containing $J$ (see Berndt and Suh \cite{BS1}). In other words, $G_2(\mathbb{C}^{m+2})$ is the unique compact, irreducible K\"{a}hler, quaternionic K\"{a}hler manifold which is not a hyper-K\"{a}hler manifold.
\par
\vskip 6pt
Let $M$ be a real hypersurface in $G_2(\mathbb{C}^{m+2})$ and $N$ a local normal unit vector field on $M$. Let also $A$ be the shape operator of $M$ associated to $N$. Then we define the structure vector field of $M$ by $\xi =-JN$. Moreover, if $\{ J_1,J_2,J_3\}$ is a local basis of ${\mathfrak J}$, we define $\xi_i=-J_iN$, $i=1,2,3$. We will call $D^{\perp} =  Span \{\xi_1,\xi_2,\xi_3\}$.
\par
\vskip 6pt
$M$ is called Hopf if $\xi$ is principal, that is, $A\xi =\alpha\xi$. Berndt and Suh, \cite{BS1} proved that if $m \geq 3$, a real hypersurface $M$ of $G_2(C^{m+2})$ for which both $[\xi]$ and ${\D}^{\perp}$ are $A$-invariant must be an open part of either (A) a tube around a totally geodesic $G_2(C^{m+1})$ in $G_2(C^{m+2})$, or (B) a tube around a totally geodesic $HP^n$ in $G_2(C^{m+2})$. In this second case $m=2n$.
\par
\vskip 6pt
Let $S$ denote the Ricci tensor of the real hypersurface $M$. In \cite{PS} we proved the non-existence of Hopf real hypersurfaces in $G_2(C^{m+2})$, $m \geq 3$, with parallel and commuting Ricci tensor, that is $\nabla S=0$ and $S{\phi}={\phi}S$ for the structure tensor $\phi$.

Moreover, this result was improved by Suh, \cite{S1} and \cite{S2}, who proved that the above non-existence property also can be hold for either parallel or commuting Ricci tensor.
\par
\vskip 6pt
In this paper, related to the parallel Ricci tensor,  we will study the corresponding condition using g-Tanaka-Webster connection. That is, we will consider real hypersurfaces for which $\hat{\nabla}_X^{(k)}S=0$ for any $X$ tangent to $M$. We obtain the following

\begin{theorem}
There do not exist connected orientable Hopf real hypersurfaces in $G_2(C^{m+2})$, $m \geq 3$, whose Ricci tensor is parallel with respect to the g-Tanaka-Webster connection.
\end{theorem}

\vskip 6pt
\par

\section{Preliminaries}\label{section 2}
\setcounter{equation}{0}
\renewcommand{\theequation}{2.\arabic{equation}}
\vspace{0.13in}

For the study of the Riemannian geometry of  $G_2(\mathbb{C}^{m+2})$ see \cite{B}. All the notations we will use from now on are the ones in \cite{BS1} and \cite{BS2}. We will suppose that the metric $g$ of $G_2(\mathbb{C}^{m+2})$ is normalized for the maximal sectional curvature of the manifold to be eight. Then the Riemannian curvature tensor $\bar{R}$ of $G_2(\mathbb{C}^{m+2})$ is locally given by

\begin{equation} \label{2.1}
\begin{split}
\bar{R}(X,Y)Z  =& g(Y,Z)X - g(X,Z)Y  +  g(JY,Z)JX - g(JX,Z)JY- 2g(JX,Y)JZ  \\                                                                                                &+  \sum_{\nu=1}^3\{g(J_{\nu} Y,Z)J_{\nu} X - g(J_{\nu} X,Z)J_{\nu}Y - 2g(J_{\nu} X,Y)J_{\nu} Z)\} \\
&+  \sum_{\nu=1}^3\{g(J_{\nu} JY,Z)J_{\nu} JX - g(J_{\nu}JX,Z)J_{\nu}JY\} ,
\end{split}
\end{equation}

\noindent where $J_1,J_2,J_3$ is any canonical local basis of ${\mathfrak J}$.
\par
\vskip 6pt

Let $M$ be a real hypersurface of $G_2(\mathbb{C}^{m+2})$, that is, a submanifold of $G_2(\mathbb{C}^{m+2})$ with real codimension one. The induced Riemannian metric on $M$ will also be denoted by $g$, and $\nabla$ denotes the Riemannian connection of $(M,g)$. Let $N$ be a local unit normal field of $M$ and $A$ the shape operator of $M$ with respect to $N$. The K\"{a}hler structure $J$ of $G_2(\mathbb{C}^{m+2})$ induces on $M$ an almost contact metric structure $(\phi,\xi,\eta,g)$. Furthermore, let $J_1,J_2,J_3$ be a canonical local basis of ${\mathfrak J}$. Then each $J_\nu$ induces an almost contact metric structure $(\phi_\nu,\xi_\nu,\eta_\nu,g)$ on $M$.
\par
\vskip 6pt
Since ${\mathfrak J}$ is parallel with respect to the Riemannian connection $\bar{\nabla}$ of $(G_2(\mathbb{C}^{m+2}),g)$, for any canonical local basis ${J_1,J_2,J_3}$ of ${\mathfrak J}$ there exist three local 1-forms $q_1,q_2,q_3$ such that

\begin{equation} \label{2.2}
\bar{\nabla}_XJ_{\nu}=q_{\nu +2}(X)J_{\nu +1}-q_{\nu + 1}(X)J_{\nu +2}
\end{equation}

\noindent for any $X$ tangent to $G_2(\mathbb{C}^{m+2})$, where subindices are taken modulo $3$.
\par
\vskip 6pt
From the expression of the curvature tensor of $G_2(\mathbb{C}^{m+2})$ the Gauss equation is given by

\begin{equation} \label{2.3}
\begin{split}
R(X,Y)Z =& g(Y,Z)X-g(X,Z)Y     \\
&+g(\phi Y,Z)\phi X-g(\phi X,Z)\phi Y-2g(\phi X,Y)\phi Z      \\
&+\sum_{\nu =1}^3 \{g(\phi_{\nu}Y,Z)\phi_{\nu}X-g(\phi_{\nu}X,Z)\phi_{\nu}Y-2g(\phi_{\nu}X,Y)-2g(\phi_{\nu}X,Y)\phi_{\nu}Z \}   \\
&+\sum_{\nu =1}^3 \{ g(\phi_{\nu}\phi Y,Z)\phi_{\nu}\phi X-g(\phi_{\nu}\phi X,Z)\phi_{\nu}\phi Y \}    \\
&-\sum_{\nu =1}^3 \{ \eta(Y)\eta_{\nu}(Z)\phi_{\nu}\phi X-\eta(X)\eta_{\nu}(Z)\phi_{\nu}\phi Y \}  \\
&-\sum_{\nu =1}^3 \{ \eta(X)g(\phi_{\nu}\phi Y,Z)-\eta(Y)g(\phi_{\nu}\phi X,Z) \} \xi_{\nu}    \\
&+g(AY,Z)ZX-g(AX,Z)AY
\end{split}
\end{equation}

\noindent for any $X,Y,Z$ tangent to $M$.

From (\ref{2.3}) the Ricci tensor of $M$ is given by

\begin{equation} \label{2.4}
\begin{split}
SX =& (4m+7)X-3\eta(X)\xi -\sum_{\nu =1}^3 \eta_{\nu}(X)\xi_{\nu}  \\
&+\sum_{\nu =1}^3 \{ \eta_{\nu}(\xi)\phi_{\nu}\phi X-\eta(\phi_{\nu}X)\phi_{\nu}\xi-\eta(X)\eta_{\nu}(\xi)\xi_{\nu} \}    \\
&+hAX-A^2X
\end{split}
\end{equation}

\noindent  for any $X$ tangent to $M$, where $h$ denotes $trace(A)$.
\par
\vskip 6pt

From (\ref{2.4}) we can compute, see \cite{PS},

\begin{equation} \label{2.5}
\begin{split}
(\nabla_XS)Y =& -3g(\phi AX,Y)\xi -3\eta(Y)\phi AX     \\
&-3\sum_{\nu =1}^3 \{ q_{\nu +2}(X)\eta_{\nu +1}(Y)-q_{\nu +1}(X)\eta_{\nu +2}(Y)+g(\phi_{\nu}AX,Y) \} \xi_{\nu}       \\
&-3\sum _{\nu =1}^3 \eta_{\nu}(Y) \{  q_{\nu +2}(X)\xi_{\nu +1}-q_{\nu +1}(X)\xi_{\nu +2}+\phi_{\nu}AX \}  \\
&+\sum_{\nu =1}^3 \{ X(\eta_{\nu}(\xi))\phi_{\nu}\phi Y+\eta_{\nu}(\xi) \{ -q_{\nu +1}(X)\phi_{\nu +2}\phi Y   \\
&+q_{\nu +2}(X)\phi_{\nu +1}\phi Y+\eta_{\nu}(\phi Y)AX-g(AX,\phi Y)\xi_{\nu} \}    \\
&+\eta_{\nu}(\xi) \{ \eta(Y)\phi_{\nu}AX-g(AX,Y)\phi_{\nu}\xi \} -g(\phi AX,\phi_{\nu}Y)\phi_{\nu}\xi   \\
&+\{ q_{\nu +1}(X)\eta(\phi_{\nu +2}Y)-q_{\nu +2}(X)\eta(\phi_{\nu+1}Y)-\eta_{\nu}(Y)\eta(AX)  \\
&+\eta(\xi_{\nu})g(AY,X)\phi_{\nu}\xi -\eta(\phi_{\nu}Y) \{q_{\nu +2}(X)\phi_{\nu +1}\xi   \\
&-q_{\nu +1}(X)\phi_{\nu +2}\xi +\phi_{\nu}\phi AX-\eta(AX)\xi_{\nu}+\eta(\xi_{\nu})AX \}   \\
&-g(\phi AY,X)\eta_{\nu}(\xi)\xi_{\nu}-\eta(Y)X(\eta_{\nu}(\xi))\xi_{\nu}-\eta(Y)\eta_{\nu}(\xi)\nabla_X\xi_{\nu} \}  \\
&+X(h)AY+h(\nabla_XA)Y-(\nabla_XA^2)Y
\end{split}
\end{equation}

\noindent for any $X,Y$ tangent to $M$, where the subindices are taken modulo 3.
\par
\vskip 6pt

A real hypersurface of type (A) has three (if $r=\frac{\pi} { 2\sqrt {8}}$) or four (otherwise) distinct principal curvatures $\alpha =\sqrt{8}\cot(\sqrt{8} r)$, $\beta =\sqrt{2} \cot(\sqrt{2} r)$, $\lambda =-\sqrt{2} \tan(\sqrt{2} r)$, $\mu =0$, for some radius $r \in (0,\frac{\pi} { \sqrt {8}})$, with corresponding multiplicities $m(\alpha)=1$, $m(\beta)=2$, $m(\lambda)=m(\mu)=2m-2$. The corresponding eigenspaces can be seen in \cite{BS1}.
\par
\vskip 6pt
A real hypersurface of type (B) has five distinct principal curvatures $\alpha =-2\tan(2r)$, $\beta =2\cot(2r)$, $\gamma =0$, $\lambda =\cot(r)$, $\mu =-\tan(r)$, for some $r \in (0,\frac{\pi} { 4})$, with corresponding multiplicities $m(\alpha)=1$, $m(\beta)=3=m(\gamma)$, $m(\lambda)=4m-4=m(\mu)$. For the corresponding eigenspaces see \cite{BS1}.

\vskip 8pt

\section{Proof of the Theorem}\label{section 3}
\setcounter{equation}{0}
\renewcommand{\theequation}{3.\arabic{equation}}
\vspace{0.13in}

If the Ricci tensor of $M$ is g-Tanaka-Webster parallel we get

\begin{equation} \label{3.1}
\begin{split}
0&=(\hat{\nabla}_X^{(k)}S)Y=\hat{\nabla}_X^{(k)}SY-S\hat{\nabla}_X^{(k)}Y    \\
&=\nabla_XSY+g(\phi AX,SY)\xi -\eta(SY)\phi AX-k\eta(X)\phi SY   \\
&\quad -S\nabla_XY-g(\phi AX,Y)S\xi +\eta(Y)s\phi AX+k\eta(X)S\phi Y
\end{split}
\end{equation}

\noindent for any $X,Y$ tangent to $M$. This yields

\begin{equation} \label {3.2}
\begin{split}
(\nabla_XS)Y =& -g(\phi AX,SY)\xi +\eta(SY)\phi AX+k\eta(X)\phi SY   \\
&+g(\phi AX,Y)S\xi -\eta(Y)S\phi AX-k\eta(X)S\phi Y.
\end{split}
\end{equation}

Thus from (\ref{2.5}) we obtain

\begin{equation} \label{3.3}
\begin{split}
&-3g(\phi AX,Y)\xi -3\eta(Y)\phi AX   \\
&-3\sum_{\nu =1}^3 \{ q_{\nu +2}(X)\eta_{\nu +1}(Y)-q_{\nu +1}(X)\eta_{\nu +2}(Y)+g(\phi_{\nu}AX,Y) \} \xi_{\nu}   \\
&-3\sum_{\nu =1}^3\eta_{\nu}(Y) \{ q_{\nu +2}(X)\xi_{\nu +1}-q_{\nu +1}(X)\xi_{\nu +2}+\phi_{\nu}AX \}    \\
&+\sum_{\nu =1}^3 \{ X(\eta_{\nu}(\xi)\phi_{\nu}\phi Y+\eta_{\nu}(\xi) \{ -q_{\nu +1}(X)\phi_{\nu +2}\phi Y  \\
&+q_{\nu +2}(X)\phi_{\nu +1}\phi Y+\eta_{\nu}(\phi Y)AX-g(AX,\phi Y)\xi_{\nu} \}   \\
&+\eta_{\nu}(\xi) \{\eta(Y)\phi_{\nu}AX-g(AX,Y)\phi_{\nu}\xi \} -g(\phi AX,\phi_{\nu}Y)\phi_{\nu}\xi    \\
&+ \{ q_{\nu +1}(X)\eta(\phi_{\nu +2}Y)-q_{\nu +2}(X)\eta(\phi_{\nu +1}Y)-\eta_{\nu}(Y)\eta(AX)    \\
&+\eta(\xi_{\nu}g(AX,Y) \} \phi_{\nu}\xi -\eta(\phi_{\nu}Y) \{ q_{\nu+2}(X)\phi_{\nu +1}\xi     \\
&-q_{\nu +1}(X)\phi_{\nu +2}\xi +\phi_{\nu}\phi AX-\eta(AX)\xi_{\nu}+\eta(\xi_{\nu})AX \}    \\
&-g(\phi AX,Y)\eta_{\nu}(\xi)\xi_{\nu}-\eta(Y)X(\eta_{\nu}(\xi))\xi_{\nu}-\eta(Y)\eta_{\nu}(\xi)\nabla_X\xi_{\nu} \}  \\
&+X(h)AY+h(\nabla_XA)Y-(\nabla_XA^2)Y    \\
&=-g(\phi AX,SY)\xi +\eta(SY)\phi AX+k\eta(X)\phi SY \\
&+g(\phi AX,Y)S\xi -€ta(Y)S\phi AX-k\eta(X)S\phi Y
\end{split}
\end{equation}

\noindent for any $X,Y$ tangent to $M$.
\par
\vskip 6pt
\begin{lemma} \label{lemma}
Let $M$ be a Hopf real hypersurface in $G_2(\mathbb{C}^{m+2})$, $m \geq 3$, such that its Ricci tensor is g-Tanaka-Webster parallel. Then either $\xi \in {\D}$ or $\xi \in {\D}^{\perp}$.
\end{lemma}

\begin{proof} We can write $\xi =\eta(X_0)X_0+\eta(\xi_1)\xi_1$, where $X_0$ is a unit vector field in ${\D}$. Suppose that $A\xi =\alpha\xi$ and that $\eta(X_0)\eta(\xi_1) \neq 0$. As $\xi(\eta_1(\xi))=g(\xi ,\nabla_{\xi}\xi_1)$, if we take $X=\xi$ and $Y=\phi X_0$ in (\ref{3.3}) we get, bearing in mind that $\eta_{\nu}(\phi X_0)=0$, $\nu =1,2,3$.

\begin{equation} \label{3.4}
\begin{split}
&3\alpha\eta(\xi_1)\eta(X_0)\xi_1+\eta(\nabla_{\xi}\xi_1)\eta(X_0)\phi_1\xi -\eta(\nabla_{\xi}\xi_1)\phi_1X_0   \\
&+\eta_1(\xi) \{ -q_2(\xi)\phi_3\phi^2X_0+q_3(\xi)\phi_2\phi^2X_0+\eta(X_0)\eta(\xi_1)\alpha\xi \}   \\
&+ \{ q_2(\xi)\eta(\phi_3\phi X_0)-q_3(\xi)\eta(\phi_2\phi X_0) \} \phi_1\xi   \\
&+ \{ q_3(\xi)\eta(\phi_1\phi X_0)-q_1(\xi)\eta(\phi_3\phi X_0) \} \phi_2\xi    \\
&+ \{ q_1(\xi)\eta(\phi_2\phi X_0)-q_2(\xi)\eta(\phi_1\phi X_0) \} \phi_3\xi   \\
&-\eta(X_0)\eta(\xi_1) \{ q_3(\xi)\phi_2\xi -q_2(\xi)\phi_3\xi -\alpha\xi_1+\alpha\eta(\xi_1)\xi \}   \\
&+\xi(\alpha)A\phi X_0+h(\nabla_{\xi}A)\phi X_0-(\nabla_{\xi}A^2)\phi X_0      \\
&=k\phi S\phi X_0+kSX_0-k\eta(X_0)S\xi.
\end{split}
\end{equation}

Now we have $S\xi =(4m+4+h\alpha -\alpha^2)\xi-4\eta(\xi_1)\xi_1$, $\eta(\phi_1\phi X_0)=-g(\phi X_0,\phi\xi_1)=\eta(X_0)\eta(\xi_1)$ and $\eta(\phi_{\nu}\phi X_0)=0$, $\nu =2,3$. Introducing these equalities in (\ref{3.4}) and taking its scalar product with $\xi$ we obtain

\begin{equation} \label{3.5}
4(\alpha -k)\eta^2(\xi_1)\eta(X_0)=g(\phi X_0,(\nabla_{\xi}A^2)\xi)-hg(\phi X_0,(\nabla_{\xi}A)\xi).
\end{equation}

As $g(\phi X_0,(\nabla_{\xi}A)\xi)=g(\phi X_0,\nabla_{\xi}\alpha\xi)=\alpha g(\phi X_0,\phi A\xi)=0$ and the same is true for the other term in the right of the equality in (\ref{3.5}) we arrive to

\begin{equation} \label{3.6}
4(\alpha -k)\eta^2(\xi_1)\eta(X_0)=0.
\end{equation}

As we suppose $\eta(\xi_1)\eta(X_0) \neq 0$, we get $\alpha =k$. Thus $\alpha$ is constant. From Berndt and Suh, \cite{BS1}, we know that $M$ being Hopf, for any $Y \in TM$, $Y(\alpha)=\xi(\alpha)\eta(Y)-4\sum_{\nu =1}^3 \eta_{\nu}(\xi)\eta_{\nu}(\phi Y)$. This yields $\phi_1\xi =0$, giving a contradiction.

Therefore $\eta(X_0)\eta(\xi_1)=0$ and we obtain the result.
\end{proof}

By Lee and Suh, \cite{LS}, if $\xi \in {\D}$, $M$ is locally a type (B) real hypersurface.
\par
\vskip 6pt
Consider the case $\xi \in {\D}^{\perp}$. Then we have

\begin{lemma} \label{lemma}
Let $M$ be a Hopf real hypersurface of $G_2(\mathbb{C}^{m+2})$, $m \geq 3$. Suppose that the Ricci tensor of $M$ is g-Tanaka-Webster parallel and $\xi \in {\D}^{\perp}$. Then $g(A{\D},{\D}^{\perp})=0$.
\end{lemma}

\begin{proof} As $\xi \in {\D}^{\perp}$ we can suppose that $\xi =\xi_1$. As $M$ is Hopf $g(A\xi_1,X)=0$ for any $X \in {\D}$. Thus we must prove that $g(A\xi_{\nu},X)=0$, $\nu =2,3$.

Take $X \in {\D}$, $Y=\xi$ in (\ref{3.3}). We obtain

\begin{equation} \label{3.7}
\begin{split}
&-3\phi AX+4g(AX,\phi_2\xi_1)\xi_2+4g(AX,\phi_3\xi_1)\xi_3+X(h)A\xi     \\
&+h(\nabla_XA)\xi -(\nabla_XA^2)\xi =\eta(S\xi)\phi AX-S\phi AX.
\end{split}
\end{equation}

From (\ref{2.4}) we have $S\xi =(4m+h\alpha -\alpha^2)\xi$ and $S\xi_2=(4m+6)\xi_2+hA\xi_2-A^2\xi_2$. If we take the scalar product of (\ref{3.7}) and $\xi_2$ and use these expressions we obtain $5g(AX,\xi_3)=6g(AX,\xi_3)$. That is

\begin{equation} \label{3.8}
g(A\xi_3,X)=0
\end{equation}

\noindent for any $X\in {\D}$. Similarly we obtain

\begin{equation} \label{3.9}
g(A\xi_2,X)=0
\end{equation}

\noindent for any $X \in {\D}$. From (\ref{3.8}) and (\ref{3.9}) the result follows.
\end{proof}

From Lemma 3.1 and Lemma 3.2 we know that $M$ is locally congruent to a real hypersurface either of type (A) or of type (B).

Suppose that $M$ is of type (A). Remember that $A\xi =\alpha\xi$, $A\xi_2=\beta\xi_2$, $A\xi_3=\beta\xi_3$, with $\alpha =\sqrt{8} cot(\sqrt{8} r)$ and $\beta =\sqrt{2} cot(\sqrt{2} r)$. Take $Y=\xi$, $X=\xi_2$ in (\ref{3.3}). We have

\begin{equation} \label{3.10}
\begin{split}
&4\beta\xi_3+h\nabla_{\xi_2}\alpha\xi -hA\phi A\xi_2-\nabla_{\xi_2}\alpha^2\xi +A^2\phi A\xi_2   \\
&=\beta \{ g(\xi_3,S\xi)\xi -\eta(S\xi)\xi_3+S\xi_3 \} .
\end{split}
\end{equation}
\par
\vskip 6pt
As $\alpha$ is constant, $S\xi =(4m+h\alpha -\alpha^2)\xi$ and $S\xi_3=(4m+6+h\beta -\beta^2)\xi_3$, from (\ref{3.10}) we arrive to $\beta\xi_3=0$, which is impossible. Thus type (A) real hypersurfaces do not have g-Tanaka-Webster parallel Ricci tensor.
\par
\vskip 6pt
In the case of a type (B) real hypersurface if we take $X=\xi_1$, $Y=\xi$ in (\ref{3.3}) and bear in mind that $S\xi =(4m+4+h\alpha -\alpha^2)\xi$ and $S\phi_1\xi =(4m+8)\phi_1\xi$, we obtain $\alpha h=0$, where $\alpha =-2\tan(2r)$. As $\alpha \neq 0$ we must have $h=0$.

Take then $X=\xi_1$, $Y=\xi_2$ in (\ref{3.3}) and bear in mind that $h=0$. With similar computations we obtain $6\beta\xi_3=0$, for $\beta=2\cot(2r)$. As this is impossible, type (B) real hypersurfaces do not have g-Tanaka-Webster parallel Ricci tensor and this finishes the proof of our Theorem.

\par
\vskip 12pt


\end{document}